 \documentclass[12pt]{amsart}

 \usepackage{pifont,geometry}
 \usepackage{amsmath,amssymb,amsthm}
 \usepackage[latin1]{inputenc}
 \usepackage[english]{babel}
 
 \usepackage{algorithm,algpseudocode}
 \usepackage{graphicx}
 \usepackage{enumitem}
 
 \usepackage[x11names,rgb,table]{xcolor}
 \usepackage{tikz}
 \usetikzlibrary{arrows,shapes}

\usepackage{fancyhdr}
 \usepackage[pdftex,bookmarks,colorlinks]{hyperref}

\newtheorem{theorem}{Theorem}[section]
\newtheorem{corollary}[theorem]{Corollary}
\newtheorem{proposition}[theorem]{Proposition}
\newtheorem{lemma}[theorem]{Lemma}

\theoremstyle{definition}
\newtheorem{definition}[theorem]{Definition}
\newtheorem{example}[theorem]{Example}
\newtheorem{remark}[theorem]{Remark}
\newtheorem{notation}[theorem]{Notation}

\numberwithin{equation}{section}

\theoremstyle{theorem}
\newtheorem*{teo}{Theorem A}

 \newcommand{\FQ}{\mathbb{F}_{q^2}}

 \newcommand{\In}{\mathrm{In}}
 \newcommand{\LM}{\mathrm{LM}}
 \newcommand{\supp}{\mathrm{Supp}}
 \newcommand{\DIM}{\mathrm{dim}}
 
 \newcommand{\He}{\mathcal{H}}
 \newcommand{\xx}{\mathcal{X}}
 \newcommand{\y}{\mathcal{Y}}

 \newcommand{\cN}{\mathcal{N}}
 \newcommand{\B}{\mathcal{B}}
 \newcommand{\MM}{\mathcal{M}}
 \newcommand{\DD}{\mathcal{D}}
 
 \newcommand{\kk}{\kappa}
 
 \def\ecr{\color{black}}
 \def\bcR{\color{black}}

\begin{document} 	
 	\title{Minimum-weight codewords of the Hermitian  codes are  supported on complete intersections}
 	
 	\author{Chiara Marcolla}
 	\address{\textnormal{Chiara Marcolla}. Dipartimento di Matematica dell'Universit\`{a} di Torino\\ 
 		Via Carlo Alberto 10, 
 		10123 Torino, Italy}
 	\email{{chiara.marcolla@gmail.com}}

 	\author{Margherita Roggero}
 	\address{\textnormal{Margherira Roggero}. Dipartimento di Matematica dell'Universit\`{a} di Torino\\ 
 		Via Carlo Alberto 10, 
 		10123 Torino, Italy}
 	\email{{margherita.roggero@unito.it}}

 	
 	
\begin{abstract} Let $\He$ be the Hermitian curve defined over a finite field $\FQ$.  In this  paper we complete the geometrical characterization of the supports of  the   minimum-weight codewords of the algebraic-geometry codes over $\He$,  started in \cite{ChiaraMargherita-fase3e4-parmin}:
if $d$ is the distance of the code, the supports are all the sets of $d$ distinct $\FQ$-points on $\He$   complete    intersection  of  two curves defined by polynomials  with prescribed initial monomials w.r.t.  \texttt{DegRevLex}.

For most Hermitian codes, and especially for all those with distance  $d\geq q^2-q$ studied in  \cite{ChiaraMargherita-fase3e4-parmin},  one of the two curves is always the Hermitian curve $\He$ itself, while  if  $d<q$ the supports are  complete intersection of two curves none of  which can be $\He$. 

Finally,  for  some special codes among those with intermediate distance between $q$  and $ q^2-q$,  both possibilities occur. 
   
We provide simple and explicit  numerical  criteria   that allow to decide for each code what kind of supports its  minimum-weight codewords have and to obtain a parametric description of the  family (or the two families) of the supports. 
\end{abstract}

\maketitle


\section{Introduction}

Let $\FQ$ be the finite field with $q^2$ elements, $q$   a power of a prime. The \textit{Hermitian curve} $\He$ is the affine, plane curve over $\FQ$ defined  by the polynomial $x^{q+1}=y^q+y$. It is a smooth curve of genus $g=\frac{q^2-q}{2}$ with only one point at infinity.
The curve $\He$ is the most studied example of  \emph{maximal curve}, that is with  the  maximum number of $\FQ$-points allowed by the Hasse-Weil  bound \cite{CGC-alg-art-rucsti94}. For this reason it is well suited for the construction of algebraic-geometric codes.

After Stichtenoth's work \cite{stichtenoth1988note} several authors give significant contribution on this topic, we quote for instance  \bcR \cite{CGC-cod-book-hirschfeld2008algebraic, CGC-cd-book-AG_HB, kirfel1995minimum,yang1992true} \ecr regarding a generic description of Hermitian codes and their distance, \cite{lee2009list,lee2010algebraic,CGC-cd-prep-manumaxchiara12,o2000decoding} about the decoding of the Hermitian codes, \cite{CGC-cod-art-ballico2012geometry,CGC-cod-art-couvreur2012dual,CGC-cod-art-fontanari2011geometry, CGC-cd-art-marcolla2014hermitian,CGC-cd-art-marcolla2015small} on small weight codewords and \cite{barbero2000weight,heijnen1998generalized,munuera1999second, yang1994weight} regarding generalized Hamming weights. \bcR Other results concerning weight hierarchies of codes on Hermitian curve and the \textit{relative} generalized Hamming weights  can be found e.g. in \cite{ballico2016higher,geil2014relative,geil2000footprints, homma2009second,lee2015bounds}. Finally,  in \cite{homma2005toward,homma2006complete,homma2006two,homma2006twob,park2010minimum} we can found the distance of two-points Hermitian code. \ecr  \\

In this work we continue the geometrical description of the support of the minimum weight codewords of the Hermitian codes started in \cite{ChiaraMargherita-fase3e4-parmin}. 
We also refer to that paper for a more detailed historical introduction and bibliography.\\
Making reference to the  H{\o}holdt, van Lint and Pellikaan work  \cite{CGC-cd-book-AG_HB}, every Hermitian code can be identified by an integer $m$, so we denote it as $C_m$, and can be classified in \textit{four phases} depending on $m$ (see Remark \ref{4fasi}). 

Our main achievement concerns the set $\MM_m$ of  minimum weight codewords.
\begin{teo} 
For any Hermitian code $C_m$,  the support of every minimum-weight codeword is of either of the following type 
\begin{enumerate}
\item[i)]   complete intersection of $\He$ and a curve $\xx$ 
\item[ii)] complete  intersection of two curves $\xx $ and $\y$, neither of which can be $\He$. 
\end{enumerate} 
If  the distance of the code is lower than $q$,   all the supports are of type ii);  for those with distance at least  $q$  the supports are all of type i), except for some special code among those  with distance $\mu q$, with $\mu\leq q$, for which both types of support are present.  
\end{teo}

A proof for Theorem A applied to \bcR the codes with $m \geq 2q^2-2q-2$ (called codes of the  III and IV phase)  is given in \cite{ChiaraMargherita-fase3e4-parmin}.
In this paper we prove  this same result in the case $m\leq 2q^2-2q-3$ (I and II phase).  Moreover, we    present a necessary and sufficient condition for a divisor to be the support of a  codeword in $\MM_m$,  by an explicit description of the families of curves $\xx$ (and respectively $\xx$ and $\y$).  \ecr 

\bcR 
The tools we exploit are very close to the classical ones (Gr\"obner basis, sous-\'escalier, etc)  that  have  already been exploited in several papers  \bcR to bound the minimum distance or to establish determinate parameters of some codes, such as \cite{geil2008evaluation,geil2000footprints,geil2017bounding}.

However, we introduce  two apparently small modifications  in the base settings, since they  allow  to greatly simplify statements and  proofs  of our results.

\smallskip

The first modification concernes the way of labeling a code. When different integers $m$ can be associaed to the same code,  we choose to label it by a special one,  taking in account  its position with respect to the  \textit{gaps} of the numerical semi-group $\Lambda=\langle q,q+1\rangle$.  In this way, the distance $d(m)$ of the  code $C_m$  and the degree of the curve $\mathcal X$ (resp. of the curves $\mathcal X$ and $\mathcal Y$) can be  expressed in a simple way. For more details on this point we refer to Subsection \ref{subsecgaps}. and to Sections 3 and 4.

\smallskip

Another change concerns  the choice of the   term  ordering. Due to the problem we deal with, we must consider curves and their intersections in the affine plane. However, as well known, the intersection theory is simpler in the projective framework. When we study the intersection of the Hermitian curve  $\mathcal H$ with another curve $\mathcal X$, we first intersect their projective closures $\overline{\mathcal H}$ and $\overline{\mathcal X}$  and then exclude the points that lie on  the infinity line $L_\infty$.  In our setting there is one and only one monomial in the equation $f(x,y)$ of $\mathcal X$ that encodes both informations: the degree of $\overline{\mathcal X}$  (hence that of $\overline{\mathcal H}\cap \overline{\mathcal X}$)  and the degree of  $\overline{\mathcal H}\cap \overline{\mathcal X}\cap L_\infty$:  this monomial turns out to be the leading monomial of $f(x,y)$ with respect to the degree reverse graded term ordering $\mathtt{DegRevLex}$ with $y>x$. 
The choice of this term ordering, instead of the one adopted in other papers on this topic, is due to this useful property  (see Section \ref{sub.TermOrder}).

\smallskip

 The discussion about codes $C_m$  with $m\leq 2q^2-2q-3$   turned out to be more difficult with respect to that concerning   larger values of $m$,   just because of the presence of the gaps of $\Lambda$ to which correspond   gaps  for  the distances of the codes and different beaviour of the supports of minimum weight codewords.

Indeed, if $m\geq 2q^2-2q-2$ , the supports of minimum weight codewords are complete intersection divisors cut on $\He$  by a single family of curves, while if $m\leq 2q^2-2q-3$ there  are  some special values of  $m$, depending on the  distribution of the gaps,  for which  there are two different type of supports:  part of the minimum weight codewords   are  supported on complete intersection divisors $\He \cap \mathcal X$, and part   on complete intersection $\mathcal Y  \cap \mathcal X$ of two curves both different from $\He$.\\

Note that there is a partial overlap with known results as far as it  concerns the geometric description of minimum weight codewords. In particular, Marcolla, Pellegrini and Sala \cite{CGC-cd-art-marcolla2015small} find a geometric description for the codes with $m\leq q^2-2$, while Ballico and Ravagnani \cite{CGC-cod-art-ballico2012geometry} find it for some codes in the range  $m\leq 2q^2-2q-3 $. In Section \ref{sec.6} we analyze in detail similarities and differences between these two papers  and the present one. 
Here we only underline that the  methodologies   we introduce allow us to describe features of the codes $C_m$,  for a large range of values $m$, by a  formula and prove them with short proofs, avoiding case by case arguments.   The previously known results follow as special cases by  the general methods  we use for  the whole classification, so that no substantial advantage or simplification would be obtained by  focusing only on the new cases.
\ecr
\smallskip

We intend to further exploit the potentiality of  these methods to also analyze the support  of small weight codewords. Some results presented in this paper, as for instance those in Section \ref{sez.phaseII},  give information on this more general case.  We are confident  that, from this strong geometric characterization, also  the explicit computation of the weight distribution will follow.

The paper is organized as follows:
\begin{itemize}
\item[-] In Section \ref{Sec.pre} we recall some basic definitions  and we prove some preliminary results about  Hermitian codes $C_m$.
\item[-] In Section \ref{sez.phaseI} we consider the codes in the I phase, computing their distance and providing the geometric description of their minimum-weight codewords, according to the results obtained in \cite{CGC-cd-art-marcolla2015small}.
\item[-] In Section \ref{sez.phaseII} we state and prove Theorem \ref{basePERparole}, which gives information about  the small-weight codewords of $C_m$ in the II phase,  and 
in Theorem \ref{basePERparole2} we  prove the distance of  these  codes.
\item[-] In Section \ref{Sec.MinWord}  we focus on the  $\FQ$-divisors that  are the support of minimum weight codewords of codes $C_m$  in the II phase, proving the main result of the present work, that is, \textit{the support of every minimum-weight codeword of these codes is a  complete intersection of two plane curves} (Theorem~\ref{curvagenericamm} and Theorem~\ref{curvagenericarifatto}).
\item[-] Finally, in Section \ref{sec.6} we compare our results with those in \cite{CGC-cd-art-marcolla2015small, CGC-cod-art-ballico2012geometry} and in Section \ref{sec.con} we draw the conclusions.
\end{itemize}

\section{Generalities and Preliminary Results}\label{Sec.pre}

\subsection{General notations}

   In the paper we fix the following notations:
   \begin{enumerate}
   \item[(i)]  $q$ is a power of a prime integer $p$ and $\FQ$ will denote the finite field with $q^2$. 
   \item[(ii)]   $\Lambda$  is the semi-group generated by $q$ and $q+1$.  For every  integer in $ \Lambda$ we always consider its  unique writing $aq+b(q+1)$ with $a\leq q$.  Every positive integer that does not belong to $\Lambda$ is called a {\it gap} of the semi-group; 
   \item[(iii)]  $m$ is an integer number $0\leq m \leq  q^3+q^2-q-2$, but we will focus mainly on the values $q^2-1 \leq m \leq  2q^2-2q-3$. 
   \item[(iv)] $\delta(m):=m-q^2+q+2$.
\end{enumerate}

\subsection{The Hermitian curve}

Let $\FQ$ be the finite field with $q^2$ elements, where $q$ is a power of a prime and let $K$ be its algebraic closure.
For any ideal $I$ in the polynomial ring $A:=\FQ[x,y]$ 
we denote by $\mathcal{V}(I)$  the corresponding  variety in $\mathbb A_K^2$. If $g_1,\ldots,g_s\in\FQ[x,y]$, we denote by $  \langle g_1,\ldots,g_s\rangle  $ the ideal they generate.\\

The \textit{Hermitian curve} $\mathcal H$ is the curve in the  affine plane $\mathbb{A}_K^2$ defined by the polynomial $H:=x^{q+1}-y^q-y$. We will denote by  $I_\He$  its defining  ideal $\langle H\rangle$ in $A$ and by   $A_\He$ the coordinate ring $A/I_\He$ of $\He$.   
\bcR  The  degree of  $\He$ is $q+1$ and its  projective closure $\overline \He$ in $\mathbb P^2_K$  is smooth, so that  the arithmetic genus and the geometric genus of $\He$ is $g:=\frac{q^2-q}{2}$.  
The number of closed  points  of $\He$ with coordinates in $\FQ$ ($\FQ$-points for short) is  $n:=q^3$:   we will always denote  them by  $P_1,\ldots,P_n$ and denote by  $E$  the zero-dimensional scheme of degree $n$ composed by them.  
The only point in  $\overline \He\setminus \He$ is $P_{\infty}=[0:0:1]$, so that $\overline \He$  has $q^3+1$
 $\FQ$-points    \cite{CGC-alg-art-rucsti94}. 
\ecr 

\begin{definition}  \label{defdivisori}  A  \textit{$\FQ$-divisor} over the Hermitian curve is a divisor $D=\sum_{i=1}^\delta Q_{i}$ where the $Q_{i}$'s are pairwise   distinct $\FQ$-points of $\He$. We will denote by   $\vert D\vert$   the degree $\delta $ of $D$. We can also write $D=\{Q_1, \dots, Q_\delta\}$; in particular,  $E=\{P_1, \dots, P_n\}$ and $D$ is  a  $\FQ$-divisor on $\He$ if and only if   $D\subseteq E$.  \\
We denote by $I_D$ the ideal generated by all polynomials in $A$  vanishing on $D$ and by $A_D $ the quotient ring $ A/I_D$.
\end{definition}

Observe that in the above notations, we have  $A_E=A/\langle H,x^{q^2}-x, y^{q^2}-y\rangle$; 
 moreover,   $D$    is a  $\FQ$-divisor on $\He$ if and only if   $A_{D}$ is a quotient of  $A_{E}$.

\subsection{\bcR Hermitain codes and term ordering \ecr}\label{sub.TermOrder}

  A linear  code $C$ over $\FQ$ is a linear subspace of $\FQ^n$ for a suitable $n$, called length  of $C$.   
The  \textit{dual code} $C^{\perp}$ of $C$ is formed by all vectors $\mathbf v$ such that $G\mathbf v^T=0$, where $G$ is the \textit{generator matrix} of $C$. Every generator matrix of $C^\perp$ is called a \textit{parity-check matrix} of the code $C$.  The \textit{weight} of a codeword $\mathbf c=(c_1, \dots,c_n)\in C$ is   the number of $c_i$ that are different from $0$ and its {\it support}   $\supp(\mathbf c)$  is the set of indexes  corresponding to the non-zero entries. The {\it distance} $d$ of a linear code $C$ is the minimum weight of its non-zero codewords. 
\\

Let us consider \textit{evaluation map} 
\begin{equation}\label{eval}
  \begin{array}{rccl}
    \phi_E :  & A_E & \longrightarrow & (\FQ)^n\\
    & f & \longmapsto & (f(P_1),\dots,f(P_n)).
  \end{array}
\end{equation}
For every given  linear subspace $V$ of $A_E $ and  ordered set of generators  $ \langle f_1, \dots, f_s\rangle$ of $V$   we can define  the  code  $C(V):=\phi_E(V)$, whose  \textit{generator matrix}   is the $s\times n$ matrix with entries  $ f_i(P_j)$, and its dual code $C(V )^{\perp}$. Note that the dual code depends on $V$, but does not depend on the chosen set of generators.

    The  Hermitian codes we consider here are  dual codes $C(V )^{\perp}$ for  the special linear subspaces $V$ defined in the following way.

\begin{definition}
 Let us fix  the weight vector $w:=[q, q+1]$ and associate to every monomial $x^ry^s$ the  $w$-degree $w(x^ry^s)=rq+s(q+1)$. 
For every  positive integer $m$,  we denote by $V_m$ the subspace of $A_E$ generated by the (classes of)  monomials of $w$-degree less than $m+1$. We will denote by $C_m$ and call   {\it{Hermitian code}} the dual code   $C(V_m )^{\perp}$.
\end{definition}

Therefore, the  length of  an Hermitian code  is the number $n$ of $\FQ$-points  in $\He$ and  the entries of a codeword $\mathbf c$ are labeled after these points, so that  we can identify the support of $\mathbf c$  with   a divisor $D\subset E$   of the points   corresponding  to the   non-zero  entries  of~$\mathbf c$.  

By definition,  for every $m<m'$ we have $C_m \supseteq C_{m'}$.

\begin{notation}  For every Hermitian code $C_m$ we will denote by $\DD_m$ the set of divisors on $\He$ that are support of codewords of $C_m$; moreover, we will denote by  $\MM_m$ the set of those that are the support of minimum-weight codewords of $C_m$.
\end{notation}

\medskip

\bcR We recall that for any given term ordering $\prec$ in a polynomial ring $S$  every polynomial
$F\in S$ has a unique \textit{leading monomial} $\LM_\prec(F)$, that is the $\prec$-largest monomial which occurs with nonzero coefficient in the expansion of $F$. If $I$ is an ideal in $S$,
 the \textit{initial ideal} $\In_\prec(I)$ as the ideal generated by the leading monomials of all the polynomials in $I$, that is 
$$\In_\prec(I) =\langle LM_\prec(F) \mid F\in I \rangle.$$\ecr
\bcR
Finally the  \textit{sous-\'escalier} $\cN(I)$ of $\In_\prec(I)$ is the set of all monomials that do not belong to $\In_\prec(I)$. Note that the quotients modulo 
the two ideals $I$ and  $\In_\prec(I)$ share the  same monomial basis $\cN(I)$.\\

\ecr

Usually the classes in $A_E$  of monomials with  $w$-degree less than $m+1$ are not linearly independent and, in order to obtain a  parity-check matrix with maximal rank,  it is convenient to   select a  suitable subset $\B$  which form a basis for $V_m$. To choose this basis we fix the \bcR \textit{graded reverse lexicographic ordering} $\mathtt{DegRevLex}$ with $y> x$, denoted by $\prec$.   We recall that the $\mathtt{DegRevLex}$ in the  polynomial ring $k[z_1, \dots, z_n]$,  with variables ordered as   $z_1< \dots <z_n$ is defined in this way: $z_1^{r_1}\cdots z_n^{r_n} \prec z_1^{s_1}\cdots z_n^{s_n}$, if ether $\Sigma r_i<\Sigma s_i$ or $\Sigma r_i=\Sigma s_i$ and  $r_i > s_i$ with $i$ minimum index   such that $r_i\neq s_i$. 
\ecr

\noindent \begin{tabular}{ccc}
\hspace{-0.2cm}\noindent\begin{minipage}{10cm}
Let $\B$ be the following set of monomials
$$
\{x^ry^s \mid 0 \leq r\leq q ,\ s\leq q^2-q-1\}\cup \{ y^s \mid  q^2-q\leq s\leq q^2-1\}. 
$$
\noindent Their classes are linearly independent in $A_E$ since they are the elements of the sous-\'escalier $\cN(I_E)$ of the    initial ideal of $I_E$ w.r.t.  $\prec$; indeed   
$\In_\prec(I_E)=\langle x^{q+1}, xy^{q^2-q},y^{q^2} \rangle$. 
\noindent With this choice,   the set of monomials  $\B_m:=\B\cap V_m$ is a basis  of $V_m$ for every $m\geq 0$. 
\end{minipage}

&
$$\quad$$
&
\begin{minipage}{4cm}
\includegraphics[width=4cm]{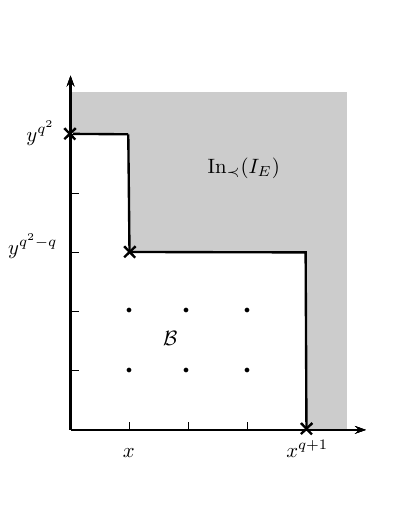}
\end{minipage}
\end{tabular}
\\


\bcR 
 
%

 \bcR
 Note that the term ordering we usually find in  literature about Hermitian codes is the weighed term ordering  $<_w$ associated to the weight vector  $w=[q ,q+1]$. However, the leading monomial  w.r.t. $<_w$  of  $H$ is the monomial $y^q$, whose degree is different from the one of $H$. More generally,   the degree of a scheme defined by an ideal $I$ cannot be computed simply considering the initial ideal of $I$  with respect to  $<_w$. 

 Furthermore, when 
 we deal with  curves in the affine plane and also with their   closure   in the projective plane,  the term ordering  $\mathtt{DegRevLex}$ does not substantially modify the leading monomial of a polynomial when we homogenize it, hence it is the more convenient  to compute  monomial bases for the coordinate ring of  an affine  scheme and that of its projective closure  through  initial ideals and   sous-escaliers.

\ecr

\begin{remark}\label{def:codice}
	For every $m\geq 0$ the  \textit{Hermitian code} $C_m$ is the code in $\FQ^n$ with parity-check matrix
	$ \phi_{E}(\B_m)$ where $\phi_{E}$ is the evaluation map (\ref{eval}) at the points of $E$. \\
	For every divisor $D$ on $\He$ we will denote by $V_{m,D}$   the image  of $V_m$ in $A_D$.  
\end{remark}

\subsection{The semi-group  $\Lambda$ and the distance of a Hermitian  code $C_m$}\label{subsecgaps}

   The semi-group $\Lambda$  is related to many features  of the Hermitian curve;  for instance  the number of gaps in $\Lambda$ is equal to  the genus $g=\frac{q^2-q}{2}$ of $\He$.  The semigroup  $\Lambda$ and its gaps play a central role  also in the study of the Hermitian codes.  We  now  collect   some  features of the gaps and some interesting relation that concern the subset   $\Lambda_\B$ of $\Lambda$ of the $w$-weights of the monomial in $\mathcal B$.

\begin{remark} \label{qualim}   
\begin{enumerate}
\item\label{qualim1gaps}  There are $q-1$      \lq\lq segments\rq\rq of gaps of decreasing length; more precisely, for every $h=1, \dots, q-1 $ there are  the following $(q-h)$ gaps: 
\begin{equation}\label{listagaps}  (h-1)(q+1)+1, \quad  \dots \quad, (h-1)(q+1)+t, \quad      \dots\quad  , (h-1)(q+1)+(q-h)
\end{equation}
with $1\leq t \leq q-h$.
 \bcR Hence, for every $h=1, \dots, q-2$, \ecr between to segments of gaps we find the following segment of elements of $\Lambda$:
    \begin{equation}\label{listanogaps} 
    hq=w(x^h), \quad   \dots\quad ,  (h-i)q+i(q+1)=w(x^{h-i}y^i), \quad \dots \quad , h(q+1)=w(y^h).
    \end{equation}
\bcR After the last gap every integer  can be written in a unique way as $$(q+t-1)q+k(q+1)=w(x^{q+t-1}y^k) \mbox{ with } t\geq 0 \mbox{ and } 0\leq k<q.$$\ecr
\item\label{qualim1} There is a $1-1$ correspondence between Hermitian codes and elements in $\Lambda_\B$. In fact,  if $m+1\notin \Lambda_\B$, then $\B_m=\B_{m+1}$, so $C_m=C_{m+1}$, while if  $m+1\in \Lambda_\B$ the codes $C_m$ and $C_{m+1}$ are different.  For this reason  in the following  {\bf we will always label  an Hermitian  code   by an integer $m$ such that} $m+1\in \Lambda_\B$.

\item\label{qualim2} The  $w$-weights of the monomials in $\B$  are pairwise different, so that there is a $1-1$ correspondence between   $\mathcal B$ and   $\Lambda_{\mathcal B}$. Moreover, the monomials in  $\B$  are ordered by  $\prec$ in the same way as by  the  weighed term order with weight-vector $w$,  namely they are ordered as   the corresponding elements in $\Lambda_{\mathcal B}$ are ordered by $<$.
\item\label{qualim3} Due to the previous items,  for every code $C_m$ we can find in $\mathcal B $ a unique  monomial whose $w$-weight is $m+1$.
\end{enumerate}
\end{remark}


In the previous paper \cite{ChiaraMargherita-fase3e4-parmin} the authors prove that the distance $d(m)$ of the Hermitian codes $C_m$ with $m\geq 2q^2-2q-2$ is given  by the integer   $ \delta(m):=  m-(q-2)(q+1)$. This formula  agrees with that proved by  \cite{CGC-cd-book-AG_HB,stichtenoth1988note,yang1992true}, taking in account our convention on the choice of the integer $m$ labeling the code $C_m$ given in Remark \ref{qualim} \eqref{qualim1}.\\

\bcR  The following  example shows that  $d(m)$ can be different from $\delta(m)$ if either  $m\notin \Lambda _{\mathcal  B} $ or for some  $m$ in the range $q^2-q-1 \leq m \leq 2q^2-2q-3$.  \ecr 
 
 \begin{example}\label{q4fase2}
For   $q=4$  let us consider the range  $10\leq m\leq 21$, containing the first integers $m$ such that $\delta(m)$ in not negative.     As already observed,  $C_{10}=C_{11}$ since $10+1 \notin \Lambda_\B$; hence $d(10)=d(11)$, while $\delta(10)\neq \delta(11)$.\\  
 Moreover, by a direct computation (or using the known formula) we   see that the distance of the Hermitian code $C_{11}$  is $d(11)=4$, while $\delta(10)=0$  and $\delta(11)=1$.\\
  More generally,  for  $m$ in the above range,  we have  $d(m)=  \delta(m)$ if  $m= 14, 15, 18, 19, 20$ and $d(m)\neq \delta(m)$  if    $m=10, 11, 12, 13, 16,  17,  21 $.  We underline that this second list contains precisely all the integers  in the considered range such that  either $m+1\notin \Lambda_{\B} $ ($m=10$)  or $\delta(m) \notin \Lambda_\B $. We will examine  this case again in Example \ref{esempiocompleto}. 
\end{example}
  \bcR 
 We now introduce a new  function   of $m$ (constructed starting from the function $\delta$);  we will prove  that it coincides with the distance $d(m)$ of every code  $C_m$ with $m\geq q^2-q-1$ (see Remark~\ref{rem.dist.I} and Theorem~\ref{basePERparole2}).
 \ecr

 \begin{notation} In the following, \bcR given $m$ \ecr we will denote by  $\widetilde{m}$ the minimum integer $m' \geq m$ such that both  $m'+1\in \Lambda_{\B}$ and $\delta(m')\in \Lambda$; moreover  $\widetilde \delta_m$ will denote the integer $\delta  (\widetilde m)=\widetilde m-(q-2)(q+1)$.
\end{notation}
 
 \bcR
 \begin{example}
 	Let us consider $q=4$ and $m=10$. Then $\widetilde m=14$ is the minimum integer $m'\geq 10$ such that $m'+1\in \Lambda_{\B}$ and $\delta(m')=m'-10\in \Lambda$. In fact, if $m'=10$ we have that $10+1\not\in \Lambda_{\B}$ and when  $m'=11,12,13$  then $m'+1\in \Lambda_{\B}$ but $\delta(m')=1,2,3\not\in \Lambda$. Whereas when $m'=14$, we have both $15\in \Lambda_{\B}$ and $\delta(m')=4\in \Lambda$, so $\widetilde m=14$.\\
 	What will see (Example \ref{solito3}) is that the code $C_{10}=C_{11}$ has distance   $\widetilde \delta_m=4$. 
 \end{example}

\bcR
Note that with our convention  about the integer   that  labels a code $C_m$,  there always exists a Hermitian code labeled after  $\widetilde m$ and that   $\widetilde \delta_m$  is  
an element in  $\Lambda_\B$. Moreover, for every $m\geq  2q^2-2q-2$ we have $m=\widetilde m$, so that  $d(m) $  agrees with $\delta  (\widetilde m)$.

\ecr
\begin{remark}\label{4fasi}
  \bcR In this paper we adopt the following classification.\ecr

\begin{enumerate}
\item[I  phase:] $m\leq q^2-2$.  For some of the integers $m$  in this range,     $m+1$ is a gap and we do not  label a code after them.  Moreover, also for some of  the remaining ones, $\delta(m) $ is a gap. 
\item[II phase:]  $q^2-1 \le m\le 2q^2-2q-3 \bcR= 4g-3\ecr$.   All the  integers $m+1$ are in $\Lambda_\mathcal B$, but some of   the $\delta(m)$ are gaps; then $m=\widetilde{m}$ if and only if $\delta(m)\in \Lambda$. 
\item[III   phase:] $\bcR 4g-2 =\ecr 2q^2-2q-2\le m \le q^3-2 \bcR=n-2\ecr$. Both $m+1\in \Lambda_\mathcal B$ and $\delta(m)\in \Lambda$, so that we always have $m=\widetilde{m}$. 
\item[IV  phase:]  $\bcR n-1\ecr = q^3 -1\le m \le q^3 +q^2-q-2 \bcR=n+2g-2\ecr$. Neither $m$, nor $\delta(m)$ are gaps, but for some   $m$ we have $m+1\notin \Lambda_{\mathcal B}$; then $m=\widetilde{m}$ if and only if $m+1\in \Lambda_{\mathcal B}$.
\end{enumerate}

\bcR 
 {\bf In the following  we fix an integer $m$   such that  $ m\le 2q^2-2q-3$.}

\medskip

This range corresponds to the I and the II phase.  Among such values of $m$  we find all the gaps of $\Lambda$;  moreover,  also  $\delta (m)$  can be negative or a gap for $\Lambda$. In correspondence  with  the gaps for $\delta (m)$ there are  \lq\lq segments\rq\rq of integers  $m$   that share  the same value $\widetilde m$.  More precisely, looking at \eqref{listagaps} and \eqref{listanogaps}, we find  for each  $h \in \{1, \dots, q-1 \}$   the following  segment   
\begin{equation}\label{listagap2}  (q+h-3)(q+1)+1,  \  \dots \  (q+h-3)(q+1)+t, \      \dots\  , (q+h-3)(q+1)+(q-h). \end{equation}
Then  $\widetilde m$ is for all of them   the first integer after the segment, namely $\widetilde m=h q+(q-2)(q+1)$, so that:
\begin{equation}\label{listagap3} \left\{\begin{array}{l}\widetilde m=w(x^{h }y^{q-2})\\
 \widetilde m+1=w(x^{h -1}y^{q-1})\\
 \widetilde \delta_m=w(x^{h}).
\end{array}\right.\end{equation}
\end{remark}

\begin{example}\label{esempiocompleto} Let us consider again the case of Example \ref{q4fase2}: 
$q=4$  and 
 $11\leq m\leq 22.$  
In the following table we summarize the behavior of the  three integers $m$, $\widetilde m$ and $\widetilde\delta_m$; in black the   segments of decreasing length 3,2,1 of the  integers $\delta(m)$ that are gaps and the corresponding $m\neq \widetilde m$.

$$\scriptsize{\begin{array}{cccccccccccccc} m &  \bf{11}&\bf{12}&\bf{13}&14 &  15 &  \bf{16}&  \bf{17}& 18 &19&20&  \bf{21}& 22 \\ 
\widetilde{m} &   14& 14&   14&  14& 15 & 18 &  18 &  18 &19&20& 22 &22 \\
 \delta(m) &  \bf{1}&  \bf{2}&  \bf{3}& 4=w(x)& {5}&  \bf{6} &  \bf{7} &   8=w(x^2) & 9&10&   \bf{11} &12=w(x^3)\\
\widetilde\delta_m &  4&  4&  4& 4& 5 &  8 &  8 &   8& 9&10&   12 &12\end{array}}$$
Note that for $11\leq m \leq 14$ we are in the I phase, while the following values  $m\leq 21 $ correspond to the II phase and  $C_{22}$ is the first code in the III phase, already considered in \cite{ChiaraMargherita-fase3e4-parmin}.
\end{example}

  \bcR 
In order to find the distance of the code $C_m$ and mainly to describe the minimum weight codewords as complete intersections,  we now study the divisors cut on  $\He$  by another  curve.  Therefore, the  following  results concerning  these intersections are crucial   in all the paper.

 We observe that the divisor cut on $\He$ by the curve defined by a polynomial $F$ coincides with that cut by every  curve  with equation  of the type $F+ GH$. Therefore, it is sufficient to consider those whose equation is $\prec$-reduced w.r.t. $H$, namely we may assume that $\partial_x(F)\leq q$. 
\ecr

\begin{proposition}\label{gradox} 
Let $F\in\FQ[x,y]$ be a polynomial such that $\partial_x (F)\leq q$ and let $\xx$ be the curve given by $F=0$. If  $\LM_\prec (F)=x^r y^s$, then 

\begin{equation} \label{gradox_i} \In_\prec(\langle H,F\rangle)=\langle x^{q+1}, x^ry^s, y^{s+q} \rangle
\end{equation}
Moreover,   the degree of the divisor $D$ cut on  $\He$ by  $\xx $ is $w(x^r y^s)=rq+s  (q+1)$. \\
More generally, if $D$ is a divisor over $\mathcal H$ and  $x^r y^s$ is any monomial in $ \In_\prec (I_D)$ with  $r\leq q$, then $\vert D\vert \leq  r q + s (q+1)$.
\end{proposition}
\begin{proof}
See Proposition 3.4 and Corollary 3.5 of \cite{ChiaraMargherita-fase3e4-parmin}.
\end{proof}

\bcR
Note that any $\FQ$-divisor $D$
contains  the support of some codeword of $C_m$ if the image  of $V_m$ in $A_D$  has dimension less than $\vert  D\vert$. More precisely:
\ecr
\begin{proposition}\label{fondamentale} 
	Let $D$ be a $\FQ$-divisor on the Hermitian curve. Then the following are equivalent 
	\begin{enumerate}
		\item\label{fondamentale_i}  $\exists \ \mathbf c \in  C_m $ with $\mathbf c\ne 0$ such that $\supp(\mathbf c)\subseteq D$;
		\item\label{fondamentale_ii}   $\DIM(V_{m,D})\leq  \vert  D\vert$  as $\FQ$ vector space;
		\item\label{fondamentale_iii} $\exists \  x^uy^v\in\cN(\In_\prec(I_D))$  such that $m+1\leq w(x^uy^v) \leq m+q+1$; 
	\end{enumerate}  
\end{proposition}
\begin{proof} 
	See Proposition 3.3 of \cite{ChiaraMargherita-fase3e4-parmin}.
\end{proof}
\ecr

\begin{corollary}\label{contiene} Let us consider two Hermitian code $C_m$ and $C_{m'}$ with $m<m'$. Then $$\DD_m \supseteq \DD_{m'} \quad  \hbox{ and  } \quad d(m)\leq d(m').$$
	If, in particular, $d(m)=d(m')$, then $\MM_m \supseteq \MM_{m'}$.
\end{corollary}

\bcR

The following example points out in which way we can apply Proposition \ref{fondamentale} and Example \ref{solito3}, that is the continue of Example \ref{solito3pre}, shows how the definition of $\widetilde m$ and the previous results allow to compute the distance of the same codes.

\begin{example}\label{solito3pre}
	Consider again the case $q=4$ and the codes $C_m$ with     $m=11,12,13,14$.\\
	If $D$ is a $\FQ$-divisor which is the support of some codewords of $C_{11}$, we know  by Proposition \ref{fondamentale} \eqref{fondamentale_iii}   that $\cN(\In_\prec(I_D))$  contains at least one of the monomials  $x^3, x^2y, xy^2, y^3$, hence $|D|\geq 4$, since  there are  at least $4$ monomials that divide each of those monomials.  Then  $d(m)\geq 4$.  
\end{example}
\ecr

\begin{lemma}\label{comodo}
Let  $D$ be a divisor over  the  Hermitian curve and let $\kappa$ be any non-negative integer. Then
$$ y^{\kappa+q}\in \cN(\In_\prec(I_D)) \Longrightarrow x^q y^\kk \in \cN(\In_\prec(I_D)). $$
\end{lemma}
\begin{proof}
We prove the equivalent statement   $x^q y^\kappa \in \In_\prec(I_D)  \Longrightarrow y^{\kappa+q}\in \In_\prec(I_D)$.\\
Let $F$ be a monic polynomial in $I_D$ such that $\In_\prec(F)=x^q y^\kk $; we can write it as $F=x^q y^\kappa+ x^{q+1}G + M$ with  $G= \sum_{i=1}^\kappa a_ix^{i-1}y^{\kappa-i} $ and $\partial M<q+\kappa$. Then  the polynomial  $S:=x F- y^\kappa H-xGH$ belongs to $ I_D$. We prove that   its leading monomial   is $y^{\kk+q}$. By  a straightforward computation we see that  
 $S=     y^{\kappa +q}+x( M +y^q G+yG) $    is a polynomial of degree $q+\kappa$, hence  $y^{\kappa +q}$ is its maximum degree with respect to  the term order $\prec$.
\end{proof}

\section{Distance and  minimum-weight codewords of    codes in the I phase}\label{sez.phaseI}

Now we focus on  the   codes $C_m$ with $m\leq  q^2-2$, such that    $m+1 $ belongs to a sequence
 $$ hq=w(x^h), \quad   \dots\quad  (h-i)q+i(q+1)=w(x^{h-i}y^i) \quad \dots \quad  h(q+1)=w(y^h)$$
 for some $1\leq h\leq q-1$.
The distance of these  codes is at least $h+1$, since for every divisor $D\in \DD_m$ the sous-\'escalier  $\cN(\In_\prec (I_D))$ contains at least one of the following monomials   $x^h$, $x^{h-1}y, \dots, y^h$, each having at least $h+1$ factors. 

\medskip

  In fact the distance of $C_m$ is equal to $ h+1$.  We obtain this equality considering the divisor $D$ complete intersection  of  a vertical line $\mathcal L$ of equation $x=u$, with $u\in \FQ$,  and the curve union of  $h+1$ horizontal lines passing through  $h+1$ $\FQ$-points of $\mathcal L \cap \He$.     

We can now describe all the divisors in $\MM_m$.  Since all the above monomials have more than $h+1$ factors, except   $x^h$ and $y^h$, for every $D\in \MM_m$ either  $y^h\in \cN(\In_\prec (I_D))$ or  $x^h\in \cN(\In_\prec (I_D))$:
\begin{itemize} 
\item[-] If $y^h\in \cN(\In_\prec (I_D))$, then $  \cN(\In_\prec (I_D))$  must be the set of divisors of $y^h$, namely $\In_\prec (I_D)=\langle x, y^{h+1}\rangle $. Thus,  we obtain precisely the divisors described above.
\item[-] If    $x^h\in \cN(\In_\prec (I_D))$, then $D\in \MM_m$   only if   $m+1\leq w(x^h)=hq$ and $  \cN(\In_\prec (I_D))$  is  the set of divisors of $x^h$, namely  $ \In_\prec (I_D)=\langle y, x^{h+1}\rangle $. Thus,   $\MM_{hq-1} $ also contains  the   divisors $D$ complete intersection of  a non-vertical line $\mathcal L'$ of equation $y=ux+v$, with $(u,v)\in \FQ^2\setminus \{ (0,0)\}$,  and the curve union of  $h+1$ vertical lines passing through  $h+1$ $\FQ$-points of $\mathcal L' \cap \He$.
\end{itemize} 

Note that this result is equivalent to that found in \cite{CGC-cd-art-marcolla2015small}. For more details see Section~\ref{sec.6}. 
\begin{example}\label{solito3} 
 \bcR Consider Example \ref{solito3pre}  where we proved that the distance of $C_m$ with     $m=11,12,13,14$ is $d(m)\geq 4$. \ecr 
On the other hand, we see that the divisor $D'$ cut on  $\mathcal H$ by the line  $x=0$  is    the support of a codeword of $C_{14}$  since  it is  $\FQ$-divisor  on $\He$ and  $\cN(\In_\prec(I_D))$ contains the monomial $y^3$ whose  $w$-weight is $15$, so, \bcR by Proposition~\ref{fondamentale} \eqref{fondamentale_ii},  $d(14)\leq 4$.
 Therefore, by Corollary \ref{contiene} we have \ecr  $4\leq d(11)\leq d(12) \leq d(13) \leq d(14) \leq 4$  and we obtain that  the distance of the codes $C_{11}$,  $C_{12}$,  $C_{13}$, and  $C_{14}$,  is  $4$.\\
 All the divisors $D\in \MM_{14}$ are such that $\In_\prec(I_D)=\langle x, y^4\rangle$. On the converse, every  divisor on $\He$ of degree $4$ such that   $\In_\prec(I_D)=\langle x, y^4\rangle$ is in $\MM_4$ if it is made of $4$ different $\FQ$-points. Therefore  the divisors $D\in \MM_{14}$ are precisely the   sets of $4$ $\FQ$-points of $\He$ with the same $x$-coordinate. They are also the set of divisors in $\MM_{12}$, $\MM_{13}$,  while in $\MM_{12}$ there are also those with $\In_\prec(I_D)=\langle y, x^4\rangle$.\\
 \bcR
 Note that  in the above cases   $m=11,12,13,14$, the distance of the code is equal both to $h+1=4$ and to $\widetilde{\delta}_m=\delta(\widetilde m)=\delta(14)=4$ (see Remark~\ref{rem.dist.I}). However, for $m=9$ we have $h=2$ so that $d(m)=3$, while $\widetilde m=14$ and $\widetilde{\delta}_m=4$.
\end{example}
\bcR
\begin{remark}\label{rem.dist.I}
	Note that for any $q^2-q-1\leq m\leq q^2-2$ we have that
	$$d(m)=\widetilde{\delta}_m=h+1=q.$$
	In fact, $m+1$ belongs to the sequence $(q-1-i)q+i(q+1)=w(x^{q-1-i}y^i) $ where $i=0,\ldots,q-1$. So the distance of $C_m$ is $h+1=q$.  \\
	Moreover, each $m$   share  the same value $\widetilde m$.  More precisely,  for any $q^2-q-1\leq m\leq q^2-3$, by \eqref{listagap3}, we find $\widetilde m=q-(q-2)(q+1)=q^2-2$ and $\widetilde{\delta}_m=w(x)=q$. Instead, for $m=q^2-2$, $\widetilde m=m$ and  $\widetilde{\delta}_m=q$.	
\end{remark}
\ecr
\section{The distance of the Hermitian codes $C_m$ in the II phase}\label{sez.phaseII}

In this section we  prove  that  the distance of  the  codes in the II phase is given by the formula $d(m)=\widetilde \delta_m$. This result  holds true also for the codes in  the final  part  of the I phase, that   we include  in our statements, and  for the codes of the III and IV phase already considered in \cite{ChiaraMargherita-fase3e4-parmin}.   Furthermore, Theorem \ref{basePERparole}  gives informations  also about  the small-weight codewords of $C_m$, more precisely for those of weight $d(m)+\kk$ for some $\kk\leq q$. Therefore we state and prove it in the widest generality, namely also including III and IV phases.

Of course, for every $m$, the value of $ \widetilde \delta_m$ agrees with that given by the well known formulas of  the distance  proved by \cite{CGC-cd-book-AG_HB,kirfel1995minimum,stichtenoth1988note,yang1992true},  provided one takes  in account the different  notation.\\

We start considering the codes with $m=\widetilde m$.

\begin{theorem} \label{basePERparole} 
 Let us consider an Hermitian  code $C_m$ with  $ m\geq  q^2-q-1$   and  $m=\widetilde m$. Let $x^u y^v\in \B$ be  a monomial with  $w(x^u y^v)=m +1+\kk$ for some $0\leq \kk \leq q$.\\
 If  $D$ is a divisor over $\mathcal H$  such that $x^u y^v\in \cN(\In_\prec(I_D))$, then 
$|D|\geq \widetilde \delta_m+\kk.$\\
More precisely, if $\delta(m)=\mu q + \lambda (q+1)$, then 
 \begin{enumerate}
	\item[\normalfont(a)] if $ \mu=0$ and  $0\leq \kk\leq q-\lambda$ then  $|D| \geq \widetilde \delta_m+ \kk(q+1-\lambda-\kk)$.
	\item[\normalfont(b)] if  $\mu=0$ and  $\kk\geq q-\lambda+1$ then $|D| \geq \widetilde \delta_m+ \kk$.
   \item[\normalfont(c)] if $ \mu\neq 0$  and    $0\leq \kk \leq \mu-1$,  then  $|D| \geq \widetilde \delta_m+\kk.$
   \item[\normalfont(d)]  if  $ \mu\neq 0$ and   $\mu\leq \kk \leq q-\lambda$, then $|D| \geq \widetilde \delta_m+(\kk-\mu)(q-\lambda-\kk)+\kk$.
   \item[\normalfont(e)] if   $ \mu\neq 0$  and  $q-\lambda+1\leq \kk \leq q$, then $|D| \geq \widetilde \delta_m+\kk.$
 \end{enumerate}
\end{theorem}
\begin{proof}  First of all we observe that the minimum $m$ that satisfies the hypotheses  is   $ q^2-2$.  Therefore, we may assume without lost in generality $m\geq q^2-2$.\\
We first assume  that $\mu=0$, so that  $\widetilde \delta_m=\lambda(q+1)$. The only  monomial in $\mathcal B$ with  $w$-weight   $m+1+\kk$ is  $x^u y^v=x^{q-\kk}y^{\lambda+\kk-1}$.

\begin{enumerate}
\item[(a)]
If  $ q-\lambda-\kk\geq  0 $, counting the monomials that divides   $x^{q-\kk}y^{\lambda+\kk-1}$ we get the bound
 $|D| \geq (q-\kk+1)(\lambda+\kk)=\widetilde \delta_m+ \kk(q+1-\lambda-\kk)\geq \widetilde \delta_m+\kk.$
\item[(b)]  If  $\lambda+\kk-q> 0$,  we observe that $\lambda+\kk-1\geq q$ and  $y^{\lambda +\kk-1}\in  \cN(\In_\prec(I_D))$; then  we apply Lemma \ref{comodo} and see that also the monomial    $x^q y^{\lambda+\kk-1-q}$ is in $ \cN(In_\prec(I_D))$.  Counting the monomials that divide either   $x^{q-\kk}y^{\lambda+\kk-1}$ or  $x^q y^{\lambda+\kk-1-q}$ we get the bound
$|D| \geq (q-\kk+1)(\lambda+\kk)+ \kk(\lambda+\kk-q)=\widetilde \delta_m+\kk.$
\end{enumerate}
 Now suppose $\delta(m)=w(x^\mu y^\lambda)$ with $\mu >0$. The monomial $x^uy^v$ having  $w$-weight $m+1+\kk$  is $x^{\mu -\kk-1}y^{\lambda+q+\kk-1}$ for $\kk=0, \dots, \mu -1$ and $x^{\mu +q-\kk}y^{\lambda+\kk-1}$ for $\kk=\mu, \dots, q$.

\begin{enumerate}
\item[(c)]

 If  $0\leq \kk \leq \mu-1$, by    Lemma \ref{comodo} we also see  that $x^q y^{\lambda+\kk-1}\in \cN(\In_\prec (I_D))$. Counting the monomials that divide either $x^{\mu-\kk-1}y^{\lambda+q+\kk-1}$ or  $x^qy^{\lambda+\kk-1}$ we get 
$|D| \geq (\mu-\kk)(\lambda+q+\kk)+(q+1-\mu+\kk)(\lambda+\kk)=\widetilde \delta_m +\kk.$

\item[(d)]  If $\mu\leq \kk \leq q-\lambda$, counting the monomials that divide $x^{\mu+q-\kk}y^{\lambda+\kk-1}$ we get
$|D| \geq (\mu+q-\kk+1)(\lambda+\kk) = \widetilde \delta_m+(\kk-\mu)(q-\lambda-\kk)+\kk  \geq\widetilde \delta_m+\kk.$ 

\item[(e)]  If $q-\lambda+1\leq \kk \leq q$, we see in the same way that  $x^q y^{\lambda+\kk-1-q}\in \cN(\In_\prec (I_D))$.  Counting the monomials that divide either $x^{\mu+q-\kk}y^{\lambda+\kk-1}$ or  $x^qy^{\lambda+\kk-1-q} $  we get
$|D| \geq (\mu+q-\kk+1)(\lambda+\kk)  + (q-\mu-q+\kk)(\lambda+\kk-q)= \widetilde \delta_m+\kk.$
\end{enumerate}
 \end{proof}

\begin{theorem} \label{basePERparole2} Let $C_m$ be any code  with \bcR$q^2-q-1\ecr\leq m \leq 2q^2-2q-3$\ecr.
Then its   distance    is 
$$
\label{formulad}d(m)= \widetilde{\delta}_m.
$$
Moreover, there exist  divisors in $\MM_m$ that are cut on $\He$  by a suitable union of lines.
\end{theorem}
\begin{proof}  
  If $m=\widetilde m$, we have already proved the result in Theorem \ref{basePERparole}. 
Then assume  $\widetilde m>m$. By Corollary \ref{contiene} it  is sufficient to prove that $d(m)\geq d(\widetilde m)=\widetilde{\delta}_m$.\\
By this same result it is also sufficient to prove  the result for the  integers $m$ such that $\delta (m)$ is the first integer in a segment of gaps. By \eqref{listagap2}  and \eqref{listagap3} there is an  integer $h $ such that  $1\leq h  \leq  q-1$ and    $m=(q+h-3)(q+1)+1 $,  $\widetilde \delta_m=h q=w(x^h)$. 

\medskip
\bcR \noindent If $h=1$, we proved in Remark \ref{rem.dist.I} that $d(m)= \widetilde{\delta}_m$.\\ \ecr
\noindent Now we assume $h\geq 2$. In this case  $m=w(x^{q}y^{h-2})\in \Lambda_\B$ and $\delta(m-1)=(h -1)(q+1)=w(y^{h-1})\in \Lambda$, so that the integer $m-1$ is one of those considered in  Theorem \ref{basePERparole} $(a)$ and $(b)$. Therefore,  we know that    $ d(m-1)=\delta(m-1)= (h -1)(q+1)$.
 Let $D$ be any divisor in $\DD_m$. By hypothesis,   $\cN(\In_\prec(I_D))$  contains a  monomial $x^uy^v$  whose  $w$-weight is  $m+1+t $ for some $t\geq 0$.  By Corollary \ref{contiene}  $D$ also belongs to $ \DD_{m-1}$    and  we can apply to it 
Theorem \ref{basePERparole} $(a)$ with $\kk:=t+1\geq 1$ and $\lambda=h-1$.  \\
If $\kk\leq q+1-h$,  we get by $(a)$ of Theorem~\ref{basePERparole}, 
$$|D|\geq {\delta}(m-1)+ \kk(q+1-\lambda -\kk)=hq+ (k-1)(q+1-h-k )\geq hq=\widetilde{\delta}_m .$$
On the other hand, if $\kk> q+1-h$ then by $(b)$ of Theorem~\ref{basePERparole}, 
$$|D| \geq \delta(m-1)+\kk\geq (h-1)(q+1)  +(q+1-h)=hq =\widetilde{\delta}_m.$$

\noindent We obtain the equality  $d(m)=\widetilde{\delta}_m$ proving the second assertion.\\
We may assume that $m=\widetilde m$ as   $\MM_{\widetilde m}\subseteq \MM_m$ (Corollary \ref{contiene}).
With this assumption, we can apply Theorem \ref{basePERparole2} and see that  the equality $|D|=\widetilde \delta_m$ can hold  only when $\kk=0$, namely when $w(x^uy^v)=m+1$.
 In particular  if  $\mu=0$, the equality can hold  only when 
 $u=q$ and $v=\lambda-1$, while  if $\mu\geq 1$ the equality can hold only  when  $u=\mu-1$ and $v=\lambda+q-1$.\\
In both cases we  obtain an $\FQ$-divisor $D'$ of degree $\widetilde \delta_m$ as the divisor cut on  $\He$ by  a  curve $\y$ union of lines, as we proved in Theorem 4.2 of \cite{ChiaraMargherita-fase3e4-parmin}.
Specifically,
\begin{itemize}
\item[-]  if  $\mu=0$, it is sufficient to chose  $\y$ as the union of $\lambda$ horizontal lines that intersect $\He$ in  $q$ $\FQ$-points;   the monomial $x^q y^{\lambda-1}$ with $w$-weight $m+1$ belongs to $ \In_\prec(I_{D'}) =\langle x^{q+1},y^{\lambda}\rangle$, hence $D'$ is the support of codewords of $C_m$ by Proposition \ref{fondamentale} .
\item[-] if $\mu\geq 1$,  it is sufficient to chose  $\y$ as the union of $\mu$ vertical and $\lambda$ horizontal lines, each intersecting $\He$   in  $q$ and,   respectively, $q+1$ $\FQ$-points; the monomial $x^{\mu-1} y^{\lambda+q-1}$ with $w$-weight $m+1$ belongs to $\In_\prec(I_{D'}) = \langle x^{q+1},x^{\mu}y^{\lambda},y^{\lambda+q}\rangle$, hence $D'$ is the support of codewords of $C_m$.
\end{itemize}
\end{proof}


\section{Geometric description of minimum weight codewords in the II phase}\label{Sec.MinWord}

In this section we focus on the  $\FQ$-divisor that  are the support of minimum weight codewords of codes $C_m$  in the II phase and  prove that they are all complete intersection of two curves and describe equations for each of them.  In fact we prove this result assuming   $q^2-q-1\leq m \leq 2q^2-2q-3$, a range that contains the II phase and also part of the I phase.   

We first consider the case of a code $C_m$ with $m=\widetilde m$.

\begin{example}\label{esempio3}
Let us consider again the same case $q=4$  of the previous examples. For  $m=18$  we have $m+1=19\in \Lambda_\B$ and $ \delta(m)=8 \in \Lambda$, so that $m=\widetilde m=18$. By Theorem \ref{basePERparole2} the distance of the  $C_{18}$ is $d(m)=\widetilde \delta_m=\delta(m)=8$.  Let $D$ be a divisor in $\MM_{18}$.  By Theorem \ref{basePERparole},  the only monomial in $\B$ with  $w$-degree $m+1=19$, namely $xy^3$, belongs to  $\cN(\In_\prec(I_D))$. Then $\cN(\In_\prec(I_D))\supseteq \{1,y,y^2,y^3, x,xy,xy^2,xy^3\}$ and the inclusion is in fact an equality, as $|\cN(\In_\prec(I_D))|=|D|=8$.  The monomial $x^2$   is not one of them, hence it belongs to $\In_\prec(I_D)$, namely there is a curve $\xx$,  defined by a polynomial $F$ with leading monomial $x^2$,  that contains $D$.\\
Again the inclusion $D\subseteq \He \cap C$ is an equality as the degree of this  zero-dimensional scheme  is $8$ (Proposition \ref{gradox}). Therefore,  $D$ is the  complete intersection of the Hermitian curve  $\He$ and the curve defined by a polynomial $F$ whose  leading monomial is the unique monomial in $\B$ with $w$-weight $8= d(m)$.\\
On the converse, let $F$ be a  polynomial with leading monomial $x^2$, namely $F=x^2-(sy+s'x+s'')$   that     cuts on $\He$ a set $D$ of $8$ points with coordinates in $\FQ$. By Proposition \ref{fondamentale}, $D\in \MM_{18}$ since the the initial ideal of $I_D=\langle x^5-y^4-y,F\rangle =\langle y^4+y-x(sy+s'x+s'')^2 ,F \rangle$ does not contain the monomial $xy^3$ with $w$-weight  $19=m+1$.
\end{example}

\begin{theorem}\label{curvagenericamm} 
Let $C_m$ be an  Hermitian with $q^2-q-1\leq m \leq 2q^2-2q-3$ and $m=\widetilde m$. \\
If    $D$  is  a $\FQ$-divisor of degree $\delta(m)=\widetilde \delta_m=\mu q+\lambda (q+1)$, then 
$D\in \MM_m$  if and only if   $D=\He\cap \xx$ with $\xx$ a  curve defined by a polynomial $F$  such that $\LM_\prec (F)=x^\mu y^{\lambda}$. 
\end{theorem}

\begin{proof}   
 Let $D$ be a $\FQ$-divisor cut on $\He$ by a curve $\xx$  defined by a polynomial $F$ such that $\LM_\prec (F)=x^\mu y^{\lambda}$. We may assume without lost of generality that  $F$ also satisfies the condition $\partial_x(F)\leq q$; in fact, if necessary,  we can substitute $F$  by its reduction modulo $H$, that  satisfies all the required conditions   $\LM_\prec (F')=x^\mu y^{\lambda}$, $\partial_x(F')\leq q$ and  $I_D=\langle H,F'\rangle $.
 Therefore we can apply Proposition \ref{gradox}  and see that  $|D|=\mu+\lambda (q+1)$ and $\cN(\In_\prec(I_D))=\cN( \langle x^{q+1}, x^\mu y^\lambda, y^{\lambda+q}\rangle)$  contains a monomial with $w$-weight  larger than $m$; more precisely it contains $x^{\mu-1}y^{\lambda+q-1}$ if $\mu>0$ and $x^{q}y^{\lambda-1}$ if $\mu=0$.\\
By Proposition \ref{fondamentale}  and Theorem \ref{basePERparole2}, $D$ belongs to  $\MM_m$,   since $\mu q + \lambda (q+1)=\widetilde \delta_m=d(m)$.

\medskip
\noindent On the other hand, let us assume that $D$ is a divisor in  $\MM_m$.
By  Theorem \ref{basePERparole2}    its degree is $\mu q+\lambda (q+1)=\widetilde \delta_m$; moreover,  by Proposition  \ref{fondamentale}  and Theorem \ref{basePERparole},  its support  contains the monomial $x^a y^b\in \B$ such that $w(x^ay^b)=m+1$. \\
We now consider the  two cases related to the integer $\mu$.
\begin{enumerate}
\item[]  If $\mu=0$, then $x^a y^b=x^q y^{\lambda -1}$.  By Theorem \ref{basePERparole} {\rm{(a)}}, the sous-\'escalier $\cN(\In_\prec(I_D))$is the set  of monomials that divides $x^q y^{\lambda -1}$. Hence  $y^\lambda$  belongs to  $\In_\prec(I_D)$, namely   there is   a polynomial $F$ in $I_D$ with leading monomial  $  y^\lambda$.

\item[] If $\mu\neq 0$, then $x^a y^b=x^{\mu-1} y^{\lambda+q -1}$. By Theorem \ref{basePERparole} {\rm{(c)}},  $\cN(\In_\prec(I_D))$  is the set of monomials that divides either $x^{\mu-1}y^{\lambda+q-1}$ or  $x^qy^{\lambda-1}$, hence there is a polynomial $F$ in $I_D$ with leading monomial  $x^\mu y^\lambda$.
\end{enumerate}
In both cases,   $I_D$ is contained in $ \langle H,F\rangle$ and, by Proposition \ref{gradox}, they are in fact  equal since   they define zero-dimensional schemes  of the same degree $d(m)=\widetilde \delta_m$.  
\end{proof}

We observe that the proof of the above result is very similar to that of the analogous results for Hermitian codes of the III and IV phases in \cite{ChiaraMargherita-fase3e4-parmin}. Before state and prove the results for the other codes  $C_m$ of the II phase,  we show by an example the different behavior that appears when   $m<\widetilde m$.
 
\begin{example} Similarly to the previous examples, we assume $q=4$. For  $m=16$ we have $\widetilde m=18$ and the distance of the code $C_{16}$ is the same as  that of $C_{18}$, namely $d(16)=d(18)=8=w(x^2)$ (Theorem \ref {basePERparole2} or Example \ref{esempio3}).\\
By Theorem \ref{curvagenericamm} the minimum codewords of $C_{18}$ are supported on complete intersection $D$ of $\He$ and a curve $\xx$  defined by a polynomial $F$ with $\LM_\prec (F)=x^2$.   By Proposition~\ref{fondamentale} we see that these divisors $D$ belongs to $\MM_{16}$. \\
On the other hand this same result allows the possibility that there are divisors $D\in \MM_{16}\setminus \MM_{18}$. \\
The monomials  $x^u y^v\in \B$ with $16+1 \leq w(x^u y^v) <18+1$ are $x^3 y$ and $x^2 y^2$.
Every divisor $D$ such that $x^2 y^2\in \cN(\In_\prec(I_D))$ has  degree at least $9$,  hence it does not belong to $\MM_{16}$. \\
Then assume that $x^3 y\in \cN(\In_\prec(I_D))$, so that  $I_D=(F,G)$ with $\LM_\prec(F)=x^4$ and $\LM_\prec(G)=y^2$.  For instance we can obtain  divisors  of this type  in the following way: every line $y-\beta$ with $\beta \in \FQ$, $\beta \neq 0$  intersect $\He$ in $5$ points $(\alpha_i, \beta) $ with   $\alpha_i\in \FQ$.  If $\beta^4+\beta=\gamma^4+\gamma$ then  also $(\alpha_i, \gamma) \in \He$; the polynomials  $F=\Pi (x-\alpha_i)$ and  $G=(y-\beta)(y-\gamma)$ define two curves as wanted.\\
Note that the divisors that are  support of minimum-weight codewords for  $C_{17}$ are so also for $C_{16}$, but the converse is not true in general. Indeed, we have $\MM_{16}\supsetneq \MM_{17}=\MM_{18}$.
\end{example}

\begin{theorem}\label{curvagenericarifatto} 
Let $C_m$ be an  Hermitian with $q^2-q-1\leq m \leq 2q^2-2q-3$ and $m<\widetilde m$, so that $\widetilde \delta_m=\mu q =w(x^\mu )$. \\
A  $\FQ$-divisor $D$ over the Hermitian curve $\He$  is the support of minimum weight codewords of $C_m$ if and only if it is a  complete intersection  of either of the following types:
\begin{enumerate}
\item[(i)]\label{curvagenerica_i} $D=\He\cap \xx$ with $\xx$  the curve defined by a polynomial $F$  such that $\LM_\prec (F)=x^\mu $;

\item[(ii)]\label{curvagenerica_alphaq}     $m=({\mu+q-3})(q+1)+1$  with $1\leq \mu \leq q-1$ and  $D=\xx \cap \y$  with $\xx$, $\y$  curves defined by polynomials $F_1$ and $F_2$ such that $\LM_\prec (F_1)=x^q$ and $\LM_\prec (F_2)= y^{\mu}$. 
\end{enumerate}
\end{theorem}
 \begin{proof}
 By Theorem \ref{basePERparole2} and Corollary \ref{contiene}  we have $d(m)=d(\widetilde m)=\mu q$ and  $\MM_{\widetilde m}\subseteq \MM_m$. Then, by Theorem \ref{curvagenericamm} all the divisors in {\rm{(i)}} are the support of codewords of $C_m$.   Therefore, it is sufficiente to prove that   $\MM_{\widetilde m}= \MM_m$  except for $m=(\mu +q-3)(q+1)+1$ and that in this case the divisors in $\MM_m\setminus \MM_{\widetilde m}$ are those described in {\rm{(ii)}}.\\
Let $m=(\mu+q-3)(q+1)+1$  and $D\in \MM_m\setminus \MM_{\widetilde m}$. By Proposition \ref{fondamentale} and Theorem~\ref{curvagenericamm}  we know that $\cN(\In_\prec(I_D))$  contains one of the   monomials with $w$-weight between $m+1$ and $\widetilde m$, namely  of the type    $x^{q-i}y^{\mu+i-2}$ with $i=1, \dots, q-\mu$.\\
Each of these monomials, except the first one, has  more than  $\mu q$ factors, hence $D\notin \MM_m$  if $x^{q-i}y^{\mu+i-2}\in \cN(\In_\prec(I_D))$ for some $i>1$.      The first monomial $x^{q-1}y^{\mu-1}$ has exactly $\mu q $ factors and its $w$-weight is  $m+1$, so that the corresponding  divisors do not belong to $\MM_{m'}$ for every $m'>m$. \\
Moreover, we obtain a divisor $D$ of degree $\mu q$  such that $x^{q-1}y^{\mu-1}\in \cN(\In_\prec(I_D))$ only if $\cN(\In_\prec(I_D))$  is precisely the set of $\mu q$  monomials that divide $x^{q-1}y^{\mu-1}$. Therefore, $ x^q, y^\mu\in \In_\prec(I_D)$. If $F_1$ and $F_2$ are polynomials in $I_D$ such that $\LM_\prec (F_1)=x^q$ and $\LM_\prec (F_2)= y^{\mu}$, then they generate $I_D$.\\
Note that   divisors $D$ as above do exist; for instance we can choose $q$ distinct elements  $\alpha_i\in \FQ$ and $\mu $ distinct elements $\beta_j\in \FQ$ such that $tr(\beta_j)=N(\alpha_i)$ for every $i,j$ and take  $F_1=\Pi_{i=1}^q (x-\alpha_i)$ and $F_2=  \Pi_{j=1}^\mu (y-\beta_j)$.
 \end{proof}

\section{Comparison with the known results}\label{sec.6}

As already underlined the results we obtain about the distance of the Hermitian codes, though formulated in a different way, agree with previous results: \bcR in 1988, Stichtenoth \cite{stichtenoth1988note}, finds a formula for the distance $d$ of $C_m$ for any $m > q^2-q-2$. A few years later,  Yang and Kumar \cite{yang1992true}   bring to completion Stichtenoth work finding  the distance of  the remaining codes  $C_m$.
In 1995, the adoption of linear algebra and the theory of semigroups
allowed Kirfel and Pellikaan \cite{kirfel1995minimum} to obtain a different and much shorter proof of the results in  \cite{yang1992true},
somehow closer to the approach presented in this manuscript. \ecr 

Furthermore, there is a partial overlap with known results for what  concerns the geometric description of minimum weight codewords. In particular, Marcolla, Pellegrini and Sala \cite{CGC-cd-art-marcolla2015small} find a geometric description for the codes in the I phase, while Ballico and Ravagnani \cite{CGC-cod-art-ballico2012geometry} find it for a few codes of the II phase. 
We state and prove again  the known results    for sake of completeness and to highlight the simplicity of our method.   Moreover,  the different description we propose for the divisors  $D\in \MM_m$  shows that also for the already know cases they are complete intersection. 

We now compare our results with the know ones.

\paragraph{I phase} In \cite{CGC-cd-art-marcolla2015small}, the authors consider Hermitian codes $C_m$ with $d(m)\leq q$, that is, $m\leq q^2-2$. 
The result obtained in \cite{CGC-cd-art-marcolla2015small} is  complete, but the description of minimum weight codewords is  slightly different from ours. For instance in Corollary 1 of \cite{CGC-cd-art-marcolla2015small} the authors describe this support for some codes that they call \textit{edge codes}   in the following way 
\lq\lq \textit{The support of a minimum-weight codeword lies in the intersection of the Hermitian curve $\He$ and a vertical line}.\rq\rq\\
%
%
  
%

The  edge codes $C_m$ with  $q^2-q\leq m\leq q^2-2$ appear among those considered in Sections 4 and 5. Specializing   our general results we get  that the distance of such a code $C_m$ is $q$ and that       $D\in \MM_m$   if and only if   $D\subset E$ and 
it is  the complete intersection of $\He$  and a curve $\xx$    defined by a polynomial $F$ such that $LM(F)= x$; $\xx$  is indeed the  vertical line of  \cite[Corollary 1]{CGC-cd-art-marcolla2015small} that we mention above.\\

A tricky case is that of the codes $C_m$ called \textit{corner codes} in  \cite{CGC-cd-art-marcolla2015small}, where it is proved that the set of points 
\lq\lq  \textit{corresponding to minimum-weight codewords lies on a same line}.\rq\rq

The corner code with  $m={q^2-q-1}$ is one of those considered in Theorem~\ref{curvagenericarifatto} {\rm{(ii)}}; in fact $m=({\mu+q-3})(q+1)+1$  with $\mu=1$.  However,  in our description  divisors   on    vertical lines and divisors    on  non-vertical lines come from the two different families described in  either (i) or (ii) of  Theorem~\ref{curvagenericarifatto} in terms of  initial  ideal. Indeed,    if  $D\subset E$, then $D\in \MM_{q^2-q-1}$    if and only if it is of either  type :  
\begin{itemize}
\item[(i)]    $D=\mathcal L\cap \He$ with $\In_\prec(I_D)= \langle x, y^q\rangle$  ($\mathcal L$ a vertical line)
\item[(ii)]    $D=\mathcal L'\cap \xx $   with $\In_\prec(I_D)= \langle y, x^q\rangle$  ($\mathcal L'$ a non-vertical line).
\end{itemize}
\ecr


\paragraph{II phase}   
For the codes of the  II phase we compare our results with  \cite[Theorem 19]{CGC-cod-art-ballico2012geometry}, that only concerns the codes $C_m$ with $q^2 < m\leq q^2+q$ corresponding to  codes having distance  $2q$,  $2q+1$ and $2q+2$.
We compare the two descriptions of the divisors $D\in\MM_m$.
%
In \cite[Theorem 19]{CGC-cod-art-ballico2012geometry} it is proved that
\begin{itemize}
\item[-] if the distance is either $2q+2$ or $2q$, then  $D\in\MM_m$ if and only if $D\subseteq E$ and it  is contained into a conic of $\mathbb{P}^2$.
\item[-] if $d=2q+1$, then  $D\in\MM_m$ if and only if $D\subset E$ and  $P_\infty\cup D$ lie on a conic of $\mathbb{P}^2$.
\end{itemize}
This description is substantially equivalent to ours,   where the the condition about $P_\infty$ is replaced by a condition on the leading monomials of the  polynomial that defines the conic. 

However, in  \cite[Theorem 19]{CGC-cod-art-ballico2012geometry} the case  with distance $2q$ and $a=q$ (in our notation  is the code $C_{q^2}$) is missing.  Indeed, this is another tricky case  of Theorem~\ref{curvagenericarifatto} {\rm{(ii)}} similar to the one quoted above; in fact $m=({\mu+q-3})(q+1)+1$  with $\mu=2$. For these codes,   the divisors  support of minimum weight codewords are of two different types.
In particular,   $D\in\MM_m$ if and only if $D\subset E$ and it is either   the complete intersection of $\He$  and a curve $\xx$   defined by a polynomial $F$ such that $LM(F)=x^2$ or $D$ is the complete intersection of the curves defined by two polynomials $F_1$ and $F_2$ such that $\LM_\prec (F_1)=x^q$ and $\LM_\prec (F_2)= y^{2}$.   

\bcR
\begin{remark}
	 Note that also for two-points Hermitian codes there are geometrical descriptions
	 regarding minimum weight codewords. In particular, in 2010, using
	 a method based on \cite{kirfel1995minimum}, Park \cite{park2010minimum} finds the minimum distance of two-points Hermitian codes and
	 obtains a geometrical characterization of minimum weight codewords as multiplications
	 of conics and lines. Three years later Ballico and Ravagnani \cite{ballico2013dual} using different techniques
	 describe the minimum-weight codewords of some two-points Hermitian codes through an
	 explicit geometric characterization of their supports.
\end{remark}
\ecr
\section{Conclusion}\label{sec.con}

With this paper the authors conclude the analysis of the supports of minimum weight codewords of the Hermitian codes,  obtaining the general result (Theorem A):

\textit{ For any Hermitian code $C_m$,  the support of every minimum-weight codeword is a  complete intersection of two plane curves. }

In the previous paper \cite{ChiaraMargherita-fase3e4-parmin} it is proved that the distance of every code $C_m$ with $m\geq 2q^2-2q-2$  (III and IV phase) is given by $d(m)=m-q^2+q+2$, hence there is a $1-1$ correspondence between codes and distances; moreover the distance of every code can be written as     $d(m)=\mu q+\lambda (q+1)$ and a set of $d(m)$ distinct $\FQ$-points $D$   is  the support  of     minimum weight codewords  of $C_m$ if and only if $D$ is cut on $\He$ by a curve $\xx$ defined by a polynomial $F$ with  leading monomial $x^\mu y^\lambda$ w.r.t. the \texttt{DegRevLex}.

In the present paper the codes $C_m$ with $m \leq 2q^2-2q-3$ (I and II phase) are considered.  There are a few    substantial  differences with respect to the codes with higher distance. The first one is that    several codes may  have  the same distance. A second difference is that the support of  minimum weight codewords of some codes $C_m$   are   complete intersections of two curves in the affine plane, but are not complete intersection in $\He$.  This happens for almost all the codes in the I phase and for some special code in the II phase

\begin{enumerate}
	\item[-]  in the I phase   the supports of minimum weight codewords of a code  $C_m$ are sets  of $d(m)$ points   on a line;   when $d(m)<q$, they cannot be obtained cutting $\He$ by a single curve, though they are complete intersection of a line and a curve of degree $d(m)$   (see \cite[Corollary~1,Proposition~2]{CGC-cd-art-marcolla2015small} or Section~\ref{sez.phaseI}).
	
	\item[-]  in the II phase,   if  $m\equiv 1 \bmod q+1$ there are two different \lq\lq families \rq\rq of  divisors on $\He$ that are support of minimum weight codewords of $C_m$. Indeed, there is an integer $\mu$, $1\leq \mu \leq q-1$,  such that  $m=(\mu+q-3)(q+1)+1$ and the distance is $d(m)=\mu q$. The support of a minimum weight codeword can be either the divisor cut on $\He$ by a curve $\xx$ given by a polynomial $F$ with $\LM_\prec (F)=x^\mu$ or the complete intersection of two curves $\xx$ and $\y$   defined by polynomials $F_1$ and $F_2$ such that $\LM_\prec (F_1)=x^q$ and $\LM_\prec (F_2)= y^{\mu}$ (see Theorem~\ref{curvagenericarifatto} {\rm{(ii)}}).
\end{enumerate}

\ecr


\section*{Acknowledgements}
\label{ack}

The authors would like to thank the anonymous referees for their interesting and useful comments which permit us to improve our presentation.

\bibliographystyle{amsplain}

\bibliography{BibChiMarg_giugno2018}

\end{document}